\documentclass{amsart}
    \usepackage{amssymb}
\usepackage{color}
\numberwithin{equation}{section} 
\newtheorem{theorem}{Theorem}[section] 
\newtheorem{proposition}[theorem]{Proposition} 
\newtheorem{lemma}[theorem]{Lemma} 
\newtheorem{corollary}[theorem]{Corollary} 
 
\newtheorem{defn}[theorem]{Definition}

\theoremstyle{definition}

\newtheorem{example}[theorem]{Example} 
\newtheorem{remark}[theorem]{Remark}

\def\rig#1{\smash{ \mathop{\,\longrightarrow\,} 
    \limits^{#1}}}
\def\swar#1{\swarrow
   \rlap{$\vcenter{\hbox{$\scriptstyle#1$}}$}}
\def\sear#1{\searrow
   \rlap{$\vcenter{\hbox{$\scriptstyle#1$}}$}}

\def\dow#1{\Big\downarrow
   \rlap{$\vcenter{\hbox{$\scriptstyle#1$}}$}}

\newcommand{\R}{\mathbb{R}}

\newcommand{\K}{\mathbb{K}}
\newcommand{\Z}{\mathbb{Z}}
\newcommand{\N}{\mathbb{N}}

\newcommand{\PP}{\mathbb{P}}

\newcommand{\OO}{\mathcal{O}}

\title{Matrices with Eigenvectors in a Given Subspace}

\author{Giorgio Ottaviani and Bernd Sturmfels }
\email{ottavian@math.unifi.it\\bernd@math.berkeley.edu}

\begin{document} 

\begin{abstract}
The Kalman variety of a linear subspace in a vector space
consists of all endomorphism that possess an eigenvector in that subspace.
We study the defining polynomials and basic geometric invariants of the Kalman variety.
\end{abstract}

\maketitle

\section{Introduction}

Let $L$ be a fixed $d$-dimensional linear subspace in the
 vector space $\K^n$, where $\K$ is an algebraically closed field.
 We are interested in the set 
$\mathcal{K}_{d,n}(L)$ of all 
$n {\times} n$-matrices $A = (a_{ij})$ with entries in $\K$ that have
a non-zero eigenvector in $L$. This set is an algebraic variety,
and our aim is to study the defining equations and geometric
properties of $\mathcal{K}_{d,n}(L)$. To this end, we write
$\, L  = \{  x \in \K^n \, : \,C \cdot x = 0 \}$
where $C$ is a fixed $(n-d) \times n$-matrix of rank $n-d$, and we form the 
 {\em Kalman matrix}
 \begin{equation}
\label{eq:kalman}
 \left(
\begin{array}{c} C\\  CA\\  C A^2\\  \vdots\\  C A^d \end{array}\right) .
 \end{equation}
The Kalman matrix (\ref{eq:kalman}) has $n$ columns and 
 $(d+1)(n-d)$ rows. The rows are grouped into $d+1$ blocks 
of $n-d$ rows, with the $i$th block being the
$(n-d) \times n$-matrix $C A^{i-1}$. It is known
that $\mathcal{K}_{d,n}(L)$ 
consists of all matrices $A$ such that the $n$ columns of the Kalman matrix (\ref{eq:kalman})
are linearly dependent. Equivalently, we have

\begin{proposition} \label{thm:kalman}
The set $\,\mathcal{K}_{d,n}(L)$ is an algebraic variety
in the matrix space  $\K^{n \times n}$. It is 
the common zero set of the $n {\times} n$-minors of the Kalman matrix~(\ref{eq:kalman}).
\end{proposition}

Proposition \ref{thm:kalman} is a variant of a classical result  in control theory, known as {\em Kalman's Observability
Condition}. A closely related formulation is the {\em Hautus Criterion} \cite{hautus}.
 A self-contained algebraic proof  of Proposition \ref{thm:kalman} appears in \cite[Theorem 2.1]{She}.
We generalize this in Theorem \ref{sing},  which implies
Proposition~\ref{thm:kalman} as its special case $s=1$.
 Since the minors of (\ref{eq:kalman}) are homogeneous polynomials, we may
regard the {\em Kalman variety}
 $\mathcal{K}_{d,n}(L)$ as a subvariety in the projective space $\PP^{n^2-1} = \PP(\K^{n \times n})$.
 The basic invariants of this projective variety are as follows:

\begin{proposition} \label{thm:degree}
The Kalman variety $\mathcal{K}_{d,n}(L)$ is irreducible, and it has codimension $n{-}d$ and degree~$\binom{n}{d-1}$ in $\PP^{n^2-1}$.
\end{proposition}

The rank condition on the Kalman matrix is well-known in control theory,
as a criterion for controllability of linear systems. One of the earliest 
references to this matrix is Corollary 5.5 in Kalman's article \cite{Kal}.
Remarkably, no algebraist seems to have followed up and conducted
the kind of study we offer here. The one possible exception we found is
the work of Osetinski{\u\i}  {\it et al.} \cite{OVV} who examine maximal minors of the
Kalman matrix from the perspective of invariant theory and combinatorics.

The irreducibility and dimension part of Proposition \ref{thm:degree}
was  proved in the 1983 PhD dissertation of Uwe Helmke, where he
used it to compute the homotopy groups of the space of controllable
pairs $(A,C)$.
The desingularization technique of our Section 4 is close to \cite{HT},
where the approach was to look at the moment map in the setting of symplectic geometry.
A point of view related to ours is also found in \cite{CHPP}.

\smallskip

To express the equations of the Kalman variety in their simplest form, we shall assume, without loss of
generality, that the given subspace $L \simeq \K^d$ is
spanned by the first $d$ standard basis vectors $e_1,\ldots,e_d$,
and we take $C = (\, {\bf 0}_{(n-d) \times d} \,| 
\,{\bf I}_{(n-d) \times (n-d)} \,)$,
a zero matrix prepended to an identity matrix.
Given any  $M \in \K^{n \times n}$ we write 
$[M]$ for the $(n-d) \times d$-matrix
obtained from $M$ by taking the last $n-d$ rows and the first $d$ columns.
With this notation, we define the {\em small Kalman matrix} to be 
\begin{equation}
\label{eq:kalman2}
 \left( \begin{array}{c} \left[A\right] \\  \left[ A^2 \right] \\  \vdots\\  \left[A^d\right] \end{array}\right) 
 \end{equation}
 This small Kalman matrix  (\ref{eq:kalman2})   has format $d(n-d) \times d$.
 The  $d \times d$-minors  of (\ref{eq:kalman2}) and the
 $n \times n$-minors of (\ref{eq:kalman}) generate
 the same ideal $I_{d,n}$ in the polynomial ring
 $\K[A] = \K[a_{11}, a_{12}, \ldots, a_{nn}]$.
 The variety of $I_{d,n}$ is now simply denoted by $\mathcal{K}_{d,n} $, which is
  $\mathcal{K}_{d,n}(L)$ for $L = \K \{ e_1, \ldots, e_d \}$.
  We conjecture that $I_{d,n}$ is a prime ideal.
 
\begin{example} \label{ex:vierzwei}
Let $n=4$, $d=2$, $L = \K\{e_1,e_2\}$. The Kalman variety 
   $\mathcal{K}_{2,4}$ is the subset of $\PP^{15}$  consisting of 
$4 {\times} 4$-matrices $A$ that have an eigenvector
with last two coordinates zero. 
 It has codimension $2$ and degree $4$.
 The small Kalman matrix~is
$$ \begin{pmatrix}
 a_{31} &                         a_{32} \\
 a_{41} &                         a_{42} \\
  a_{11} a_{31}+a_{21} a_{32}+ a_{31}a_{33}+a_{34}a_{41} &
  a_{12} a_{31}+ a_{22} a_{32} + a_{32} a_{33} +a_{34} a_{42} \\
  a_{11} a_{41} + a_{21}a_{42} + a_{31} a_{43} +a_{41} a_{44} & 
  a_{12} a_{41} + a_{22} a_{42}  + a_{32} a_{43} +a_{42} a_{44}
  \end{pmatrix}
  $$
  The six $2 {\times} 2$-minors of this matrix
(one quadric, four cubics and one quartic)
generate the ideal $I_{2,4}$ but not minimally.
To generate $I_{2,4}$, 
the quadric and certain three of the four cubics suffice. We note that
 $I_{2,4}$ is the prime ideal of $\mathcal{K}_{2,4}$.
 \qed
\end{example}

\begin{example} \label{ex:fuenfdrei}
Let $n=5$ and $d=3$.  The Kalman variety $\mathcal{K}_{3,5}$ 
has codimension $2$ and degree $ 10$. It consists of 
$5 {\times} 5$-matrices $A$ that have an eigenvector
with vanishing last three coordinates. The small Kalman matrix 
has size $6 {\times} 3$ and it~equals
\begin{equation}
\label{eq:kalman63}
\begin{pmatrix}
A_{21} \\
A_{21} A_{11} + A_{22} A_{21} \\
A_{21} A_{11}^2 + A_{22} A_{21} A_{11} + A_{21} A_{12} A_{21}  + A_{22}^2 A_{21} 
\end{pmatrix}
\end{equation}
where we use the block decomposition
\begin{equation}
\label{eq:blocks}
 A \,\, = \,\, \begin{pmatrix} A_{11} & A_{12} \\ A_{21} & A_{22} \end{pmatrix}  
\end{equation}
 with $A_{11} $ of size $3 {\times} 3$, $A_{12} $ of size $3 {\times} 2$,
$A_{21}$ of size $2 {\times} 3$, and $A_{22} $ of size $2 {\times} 2$.
The ideal $I_{3,5}$ is prime. It is  minimally generated by
eight of the  twenty $3 {\times} 3$-minors of (\ref{eq:kalman63}), namely,
two quartics, two of the four quintics and four
of the eight sextics.
\qed
\end{example}

We now introduce one further simplification in the representation 
of the Kalman variety $\mathcal{K}_{d,n}$. Since
the condition that $A$ have an eigenvector in 
$L = \K\{e_1, \ldots, e_d\}$ does not depend 
on the entries in the last $n-d$ rows of  $A$, we can 
set $A_{12} = A_{22} = 0$ in the small
Kalman matrix. This simplifies the
$d {\times} d$-minors, but does not change the ideal
$I_{d,n}$ they generate. It is instructive to check this in
Examples  \ref{ex:vierzwei} and \ref{ex:fuenfdrei}.

We define the 
{\em reduced Kalman matrix} to be the $d(n-d) \times d$-matrix
\begin{equation}
\label{eq:kalman3}
\left(
\begin{array}{c}A_{21}\\  A_{21} A_{11} \\ \vdots \\  A_{21} A_{11}^{d-1}\end{array}\right) 
 \end{equation}
 where $A_{11} $ is the upper  left $d {\times} d$-block
 and $A_{21}$ is the lower left $(n-d) {\times} d$-block of $A$.
 Thus the reduced Kalman matrix is the small Kalman matrix with
 $A_{12} = A_{22} = 0$.
 
 In all practical computations with the ideal $I_{d,n}$
 we always use the $d {\times} d$-minors
 of the reduced Kalman matrix as the set of generators.
 For instance, the ideal $I_{3,5}$ in Example \ref{ex:fuenfdrei}
 is generated by the maximal minors of
  the $6 {\times} 3$-matrix $\begin{pmatrix} A_{21} \\ A_{21} A_{11} \\ A_{21} A_{11}^2 \end{pmatrix}$,
 where $A_{11}$ is a $2 {\times} 2$-matrix
 of variables and $A_{21}$ is a $3 {\times} 2$-matrix of variables.
  
  \smallskip
  
  This article is organized as follows. 
  A self-contained proof of
Proposition \ref{thm:degree} is given in Section 2.
 In Section 3 we show that the generators of $I_{d,n}$ form a Gr\"obner basis 
 when $d=2$.  This seems to be no longer the case for $d \geq 3$. 
 Section 4 concerns the singular
 locus of the Kalman variety $\mathcal{K}_{d,n}$, and this leads us to 
  varieties defined by non-maximal minors of (\ref{eq:kalman3}).
  In Section 5 we introduce
 vector bundle techniques for studying $\mathcal{K}_{d,n}$, 
 and we compute the degrees of the singular locus.
 
 \smallskip 
 
 This project began in June 2010, after a visit by Bernd Sturmfels to
the Basque Center of Applied Mathematics in Bilbao, where Enrique Zuazua
told him about the Kalman observability condition and its role for partial differential
equations. Beauchard and Zuazua  \cite{BZ} had shown that
the Kalman variety arises as the failure locus of the 
Shizuta--Kawashima condition for partially dissipative linear hyperbolic systems.
One week after Bilbao, Bernd visited Florence where he discussed
this topic with Giorgio Ottaviani and this led to the present paper.
We are grateful to Enrique for getting us started and for participating
in our early e-mail discussions.  
We also thank Uwe Helmke and Raman Sanyal for helping us with references.

Our result on
the Kalman variety
have the potential of being useful for a wide range of applications
other than classical control theory and PDE. One example is the
work of Tran \cite{Tran} on statistical ranking. 
In her analysis of Saaty's analytical hierarchy process \cite{Saa1, Saa2}, Tran examines
the following question: Given a positive real $n {\times} n$-matrix $A$, under
which condition  does the dominant (Perron-Frobenius)
eigenvector of $A$ fail to have distinct entries. The algebraic
condition for the $i$th and $j$th coordinate of that eigenvector
to be equal is precisely the Kalman variety of the hyperplane
$\,L \,=\, \{ x \in \R^n: x_i = x_j\}$. More generally, if $L$ is
any linear space defined by  equations of the form $x_i = x_j$,
then the Kalman variety of $L$ arises in the 
stratification of matrix space $\R^{n \times n}$ that represents
distinct outcomes in \cite{Saa2}.

\section{Dimension and degree}

In this section we present a proof of
Proposition \ref{thm:degree}. We begin with an example.

\begin{example} \label{ex:d=1}
Consider the case $d=1$ 
when $L = {\rm ker}(C)$ is spanned by a single vector $v \in 
\K^n \backslash \{0\}$. Here, the Kalman variety
$\mathcal{K}_{1,n}(L)$ is a linear space of codimension $n-1$.
Its ideal $I_{1,n}(L)$ is generated by the
$2 {\times} 2$-minors of the $n {\times} 2$-matrix
$\,\bigl(\, v \,, \, Av\, \bigr)$.
That same ideal of linear forms in $\K[A]$ is generated, highly redundantly,
by the  $\binom{2n -2}{n}$  maximal minors 
of the Kalman matrix 
$\begin{pmatrix} C \\ CA \end{pmatrix} $,
which has format  $(2n-2) \times n$. Clearly, it suffices
to take only those $n-1$ maximal minors
that involve the first $n-1$ rows and are hence linear in $A$.
If we take $v = e_1$  then both descriptions yield the ideal
$\,I_{1,n} = \langle a_{21}, a_{31}, \ldots, a_{n1} \rangle\,$
whose generators are the entries of the reduced Kalman matrix $[A]$
of format $ (n-1) \times 1$. \qed
\end{example}

\begin{proof}[Proof of Proposition~\ref{thm:degree}]
We regard  vectors $v \in \K^n\backslash \{0\}$ as
points in the projective space $\PP^{n-1}$, and  we regard
non-zero $n {\times} n$-matrices
as points in $\PP^{n^2-1}$.
The product of these two projective spaces, $\,X= \PP^{n^2-1} \times \PP^{n-1}$,
has the two projections
\begin{equation}
\label{eq:twoprojections}
\begin{array}{ccccc}&&X\\
&\swar{p}&&\sear{q}\\
\PP^{n^2-1}&&&&\PP^{n-1}
\end{array}
\end{equation}
We fix the incidence variety
$\, \mathcal{I}\,:=\,\{(A,v) \in \PP^{n^2-1} \times \PP^{n-1}\,|\,
v\textrm{\ is a eigenvector of\ }A\}$.
The projection of $\mathcal{I}$ to the second factor
$$\begin{array}{ccc} \mathcal{I} &\rig{q}&\PP^{n-1}\\
(A,v) & \mapsto & v
\end{array}$$
is surjective and every fiber is a linear subspace of codimension $n-1$ in
$\PP^{n^2-1}$. 
To be precise, the fiber over $v$ is the linear space
$\mathcal{K}_{1,n}(\K v)$ seen in Example \ref{ex:d=1}.

The projection of the incidence variety $\mathcal{I}$ to the first factor
$$\begin{array}{ccc} \mathcal{I} &\rig{p}&\PP^{n^2-1}\\
(A,v)&\mapsto&A
\end{array}$$
is surjective and the general fiber is finite. It consists of the $n$ eigenvectors of $A$.
These properties imply that
 $\mathcal{I}$ is irreducible and has codimension $n-1$ in~$X$.
 
The Kalman variety has the following description
in terms of the diagram (\ref{eq:twoprojections}):
\begin{equation}
\label{projkalman}
\mathcal{K}_{d,n}(L)
\,\,= \,\, p_*\left( \mathcal{I} \cap q^*(\PP{{}}(L))\right)
\end{equation}
The $(d{-}1)$-dimensional subspace $\PP(L)$ of $\PP^{n-1}$
specifies the following diagram:
\begin{equation}
\label{eq:twomoreprojections}
\begin{array}{ccccc}&& \mathcal{I} \cap q^*(\PP{{}}(L)) \\
&\swar{p}&&\sear{q}\\
\mathcal{K}_{d,n}(L)&&&&\PP{{}}(L)
\end{array}
\end{equation}
As before, each fiber of
the map $q$ in (\ref{eq:twomoreprojections}) is a linear space
of codimension $n-1$ in $\PP^{n^2-1}$. This implies that
$\, \mathcal{I} \cap q^*(\PP{{}}(L))\,$ is irreducible and
its dimension equals
$$   {\rm dim}(\PP(L)) + {\rm dim} (\PP^{n^2-1}) - (n-1)  \,\,= \,\,
(d-1) + n^2-1 - (n-1)  \,\,=\,\, n^2-1 -(n-d). $$
Since the general fiber of the surjection $p$ is finite, 
the Kalman variety $\mathcal{K}_{d,n}(L)$ is irreducible
of the same dimension. Hence $\mathcal{K}_{d,n}(L)$
has codimension $n{-}d$ in~$\PP^{n^2-1}$.

We next derive the degree of $\mathcal{K}_{d,n}(L)$.
Consider the prime ideal generated by the $2 {\times} 2$-minors
of a $2 {\times} n$-matrix $\,\bigl(\, u \,, \, v \, \bigr)$.
This ideal lives in a $\N^2$-graded polynomial ring
$\K[u,v]$ in $2n$ unknowns.
Its {\em multidegree}, in the sense of \cite[\S 8.5]{MS}, is equal to
\begin{equation}
\label{eq:bidegree1}
 \mathcal{C}(t_1,t_2) \,\, = \,\,\,
t_1^{n-1} + 
t_1^{n-2} t_2 + 
t_1^{n-3} t_2^2 +  \cdots +
t_1 t_2^{n-2}  + 
t_2^{n-1} .
\end{equation}
This formula is the  case $k=2$ and $\ell = n$ of 
Exercise 15.5(b) in \cite[page 308]{MS}.

The prime ideal of the incidence variety
$\mathcal{I} \subset \PP^{n^2-1} \times \PP^{n-1}$ lives in 
the $\N^2$-graded polynomial ring $\K[A,v]$ in $n^2+n$ unknowns,
namely the entries of $A$ and of $v$. This prime ideal
is generated by the $2 {\times} 2$-minors of the matrix 
$\,\bigl(\, A v \,, \, v\, \bigr)$, so it is obtained from
the ideal in the previous paragraph by
the bilinear substitution $\,u \mapsto A v $.
This implies that the multidegree of $\mathcal{I}$ is obtained
from  (\ref{eq:bidegree1}) by the substitution
$t_1 \mapsto t_1 + t_2$.
 Hence the multidegree of $\mathcal{I}$ is 
the bivariate polynomial
\begin{equation}
\label{eq:bidegree2}
\mathcal{C}(t_1+t_2,t_2) \,\, = \,\,\,
\sum_{d=1}^n (t_1+t_2)^{n-d} t_2^{d-1}
\,\,\, = \,\,\,
\sum_{d=1}^n \binom{n}{d-1} t_1^{n-d} t_2^{d-1}
\end{equation}
In geometric language, the bidegree (\ref{eq:bidegree2}) is the
{\em equivariant cohomology class} of the incidence variety $\mathcal{I}$
with respect to the natural $(\K^*)^2$-action on
 $\PP^{n^2-1} \times \PP^{n-1}$.
 
The classical degree of a projective variety 
is the number of intersection points with a general
linear subspace of complementary dimension.
Likewise, each coefficient of the bidegree 
(\ref{eq:bidegree2}) is the number of points in the
intersection of $\mathcal{I}$ with an
$(n-1)$-dimensional variety $\PP(K) \times \PP(L)$
where $\PP(K)$ and $\PP(L)$ are general linear subspaces
of $\PP^{n^2-1}$ and $\PP^{n-1}$. Specifically,
if we fix $\PP(L)$ of dimension $d-1$ then this
number equals $\binom{n}{d-1}$. Equivalently, the number
of points in the intersection of $\mathcal{K}_{d,n}(L)$
with a general $(n-d)$-dimensional linear subspace 
$\PP(K)$ of $\PP^{n^2-1}$ is equal to  $\binom{n}{d-1}$.
Therefore, $\, {\rm degree}(\mathcal{K}_{d,n}(L)) = \binom{n}{d-1}$,
and the proof is complete.
\end{proof}

\begin{example}
Consider the special case $d = n-1$ when $L$ is the hyperplane in $\K^n$
perpendicular to a non-zero row vector $C$.
Then (\ref{eq:kalman}) is the 
$n {\times} n$-matrix whose $i$-th row vector is $C A^{i-1}$. The Kalman variety 
$\mathcal{K}_{n-1,n}$ is the hypersurface in $\PP^{n^2-1}$ of degree $\binom{n}{2}$ defined 
by the determinant of that matrix. 
Likewise, the reduced Kalman matrix (\ref{eq:kalman3}) is the $(n{-}1) \times (n{-}1)$-matrix
whose $i$-th row vector equals $A_{21} A_{11}^{i-1}$. Its
determinant is that same generator of the
principal ideal $I_{n-1,n}$.
\qed
\end{example}

\section{Combinatorics for $d=2$}

In this section we fix $d=2$. 
We present a study of the ideal $I_{2,n}$
from the perspective of combinatorial commutative algebra \cite{MS}.
The reduced Kalman matrix~is 
\begin{equation}
\label{eq:kalman4}
\,\begin{pmatrix} A_{21} \\ A_{21} A_{11}  \end{pmatrix}
\end{equation}
where $A_{11}$ is the upper left $2 {\times} 2$-block 
and $A_{21}$ is the lower left $(n-2) \times 2$-block of
the $n {\times} n$-matrix $A = (a_{ij}) $. The Kalman variety 
$\mathcal{K}_{2,n}$ has codimension $n-2$, and the
Kalman ideal $I_{2,n}$ is generated by the  $2 {\times} 2$-minors of 
the $(2n-4) \times 2$-matrix~(\ref{eq:kalman4}).

This construction has the following
geometric interpretation.
Consider the Segre variety $\PP^1 \times \PP^{n-1} \subset \PP^{2n-1}$
that is cut out by the $2 {\times} 2$-minors of the $n \times 2$-matrix
$\, \begin{pmatrix} A_{11} \\ A_{21} \end{pmatrix}$.
Take the cone over that Segre variety from the  point
$\begin{pmatrix} {\bf I}_{2 \times 2} \\ {\bf 0}_{n-2 \times 2} \end{pmatrix}$.
That cone is a subvariety of codimension $n-2$
and degree $n$ in $\PP^{2n-1}$ defined by the $2 {\times} 2$-minors of (\ref{eq:kalman4}).
We conclude the following geometric description of our variety:

\begin{remark} \label{itsacone}
The Kalman variety $\mathcal{K}_{2,n} \subset \PP^{n^2-1}$
is the cone with base $\PP^{n^2-2n}$ over a
general linear projection of the Segre variety 
$\PP^1 \times \PP^{n-1}$ from $\PP^{2n-1}$ into $\PP^{2n-2}$.
\end{remark}

Our goal in this section is to prove the following result about $\mathcal{K}_{2,n}$
and its ideal.

\begin{theorem} \label{thm:gb}
The ideal $I_{2,n}$ is prime. Its minimal generators are the $\binom{n-2}{2}$~quadrics
\begin{equation}
\label{eq:gbquadrics}
 \underline{a_{i1} a_{j2} } - a_{i2} a_{j1}=\begin{vmatrix}
a_{i1}&a_{i2}\\
a_{j1}&a_{j2}\end{vmatrix} \qquad 
\text{where} \,\,\, \, 3 \leq i < j \leq n 
\end{equation}
and the $\binom{n-1}{2}$ cubics
\begin{equation}
\label{eq:gbcubics}
\!\! \begin{matrix}  &
 \underline{a_{11} a_{i2} a_{j1}} - a_{12} a_{i1} a_{j1} + a_{21} a_{i2} a_{j2} - a_{22} a_{i1} a_{j2} 
 \quad \quad \\
& \qquad = \quad -\,\begin{vmatrix}
a_{i1}&a_{i2}\\
a_{j1}a_{11}+a_{j2}a_{21}&a_{j1}a_{12}+a_{j2}a_{22}\end{vmatrix}
\end{matrix} \quad
\text{where $3 \leq i \leq j \leq n$.}
 \end{equation}
These form a Gr\"obner basis of $I_{2,n}$ with respect
to the lexicographic term order.
The Hilbert series of the quotient ring $\K[A]/I_{2,n}$ with the standard $\N$-grading equals
\begin{equation}
\label{eq:hs1}
\frac{n}{(1-z)^{n^2-n+2} } \,-\,
\frac{n-1}{(1-z)^{n^2-n+1}}
 \,- \,\frac{1}{(1-z)^{n^2-2n+4}}
  \,+\,\frac{1}{(1-z)^{n^2-2n+3}} 
\end{equation}
\end{theorem}

\begin{proof}
The listed quadrics and cubics are among the $2 {\times} 2$-minors 
of the reduced Kalman matrix, so they lie in the ideal $I_{2,n}$. Their
lexicographic leading terms are the underlined monomials.
Let $M$ denote the radical ideal generated by these monomials. 
Then we have $\, M \subseteq {\rm in}_{\rm lex} (I_{2,n})$.
Our goal is to show that equality holds.

The radical monomial ideal $M$ is the intersection of $n$ 
monomial primes, each having codimension  $n-2$.
The corresponding simplicial complex $\Delta$ has the facets
\begin{equation}
\label{eq:facets}
 \begin{matrix} & \{ a_{11}, a_{31}, a_{41}, \ldots, a_{n1} \} \,\qquad \qquad & \\
& \,\,\, \quad \{ a_{32}, a_{42} ,\ldots, a_{i2}, a_{i1} ,a_{i+1,1} ,\ldots a_{n1} \} &
\hbox{for} \,\,\, i=3,4,\ldots,n, \\
\hbox{and} & \{a_{11} , a_{32}, a_{42}, \ldots, a_{n2} \} \qquad \qquad
\end{matrix}
\end{equation}
Here we only list variables that appear in the generators of $M$,
not the $n^2 - 2n + 3$ variables $a_{ij}$ that do not appear. 
The simplicial complex that represents the variety of $M$ in $\PP^{n^2-1}$
is obtained from  (\ref{eq:facets}) by taking the join 
with a simplex of dimension $n^2-2n+2$.
From the list in (\ref{eq:facets}) we see that $M$ is unmixed 
ideal of codimension $n-2$ and degree~$n$.
Recall that being {\em unmixed} means that
$M$ is the intersection of prime ideals
each generated by a set of variables $a_{ij}$ of the same cardinality.

Since $I_{2,n}$ has codimension $n-2$ and degree $n$,
by Proposition \ref{thm:degree},
the initial monomial ideal ${\rm in}_{\rm lex}(I_{2,n})$ also has codimension $n-2$ and degree $n$.
But this implies 
$\,M = {\rm in}_{\rm lex} (I_{2,n})$. Indeed, if the unmixed radical ideal
$M$ were strictly contained in ${\rm in}_{\rm lex}(I_{2,n})$
then the latter would have lower degree or higher codimension.

This proves that the quadrics in (\ref{eq:gbquadrics})
and cubics in (\ref{eq:gbcubics}) form a Gr\"obner basis for $I_{2,n}$,
and in particular they generate $I_{2,n}$. This generating set is  minimal because the generators
are $\K$-linearly independent and
 none of the cubics can be written as
a  $\K[A$]-linear combination of the quadrics.
We have shown that ${\rm in}_{\rm lex}(I_{2,n})$
is a radical ideal, and this implies that $I_{2,n}$ is radical.
Since the variety $\mathcal{K}_{2,n}$ defined by $I_{2,n}$ is irreducible,
we can conclude that the ideal $I_{2,n}$ is prime.

Let $\Delta'$ be the simplicial complex  of dimension $n-2$
obtained from $ \Delta$ by replacing the occurrence of
the vertex $a_{11}$ in the last facet of (\ref{eq:facets})
with a different vertex $a_{11}'$. Then the given ordering
of the $n$ facets represents a {\em shelling} of $\Delta'$. In that shelling order,
precisely one new vertex enters whenever a new facet gets attached.
This implies that the Hilbert series of the Stanley-Reisner ring
$\K[\Delta'] $ is equal to
$\, \bigl(1 + (n - 1) z \bigr)/(1-z)^{n-1}$.
When we pass from $\Delta'$ to $\Delta$
then we identify $(a_{11}')^i$ with $a_{11}^i$ for each $i \leq 1$,
and all other standard monomials remain unchanged.
Hence we lose one monomial per positive degree,
so we need to subtract $z/(1-z)$ from the previous Hilbert series.
This implies that  $\K[\Delta] = \K[A]/M$, regarded as a quotient
of the polynomial ring in $n^2$ variables, has the Hilbert series
\begin{equation}
\label{eq:hs2}
\frac{1}{(1-z)^{n^2-2n+3}} \biggl( \frac{1 + (n - 1) z }{(1-z)^{n-1}} - \frac{z}{1-z} \biggr) 
\end{equation}
This is also the Hilbert series of $\K[A]/I_{2,n}$. It equals
 the rational function (\ref{eq:hs1}).
\end{proof}

From the Hilbert function formula (\ref{eq:hs1}) 
we can read off the Hilbert polynomial:

\begin{corollary}
The Hilbert polynomial of the Kalman variety $\mathcal{K}_{2,n}$
in $\PP^{n^2-1}$ equals
\begin{equation}
\label{eq:easyeqn}
 n \bigl[\PP^{n^2-n+1}\bigr] \, - \,(n- 1) \bigl[\PP^{n^2-n} \bigr]\ \,-\, 
\bigl[\PP^{n^2-2n+3}\bigr] \,+\,\bigl[\PP^{n^2-2n+2}\bigr], 
\end{equation}
where $\bigl[\PP^r\bigr] =\binom{t+r}{r}$ is the Hilbert polynomial of
$r$-dimensional projective space $\PP^r$.
\end{corollary}

\begin{remark}
The Gr\"obner basis 
of quadrics
and cubics in Theorem \ref{thm:gb}
is not reduced since
the last term  $ a_{22} a_{i1} a_{j2} $
 of (\ref{eq:gbcubics}) for $i < j$
is divisible by the leading term of (\ref{eq:gbquadrics}).
Replacing that last term by 
$ a_{22} a_{i2} a_{j1} $, we get 
the reduced Gr\"obner~basis.
\end{remark}

Theorem~\ref{thm:gb} does not extend to $d \geq 3$.
In particular,
the lexicographic initial ideal ${\rm in}_{\rm lex}(I_{3,5})$ is not
radical. At present we do not even know whether
$I_{3,n}$ is prime and we do not have a general
formula for its Hilbert polynomial. 

\begin{example}\label{n9d3}
The Kalman variety  for  $n = 9$ and $d = 3$  has
codimension $6$ and degree $36$ in $\PP^{26}$.
Its {\em K-polynomial} (i.e.~the numerator of its Hilbert series) equals
$$ \begin{matrix}
(1-z)^6 \cdot (1+6z+21z^2+36z^3-19z^4-109z^5+151z^6+152z^7-637z^8\\
 \qquad \quad \,\, +\,896z^9-792z^{10}   +495z^{11}-220z^{12}+66z^{13}-12z^{14}+z^{15}),
  \end{matrix}
$$
and its Hilbert polynomial is found to be
\begin{small}
$$
 36[\PP^{20}]-63[\PP^{19}]+28[\PP^{18}]-42[\PP^{14}]+111[\PP^{13}]-97[\PP^{12}]+28[\PP^{11}]
 +[\PP^8]-3[\PP^7]+3[\PP^6]-[\PP^5].
$$
 \end{small}
 In particular,  for $d \geq 3$  there is no formula as short and simple as (\ref{eq:easyeqn}). \qed
\end{example}

The following result  about the generators of $I_{3,n}$
appeared as a conjecture in the first version of this paper.
A proof was announced by Steven Sam in May 2011.

\begin{theorem}\label{3generators}{\rm (Sam \cite{Sam})}
The ideal $I_{3,n}$ is minimally generated by polynomials of
degree $3, 4, 5, 6$. There are
$\binom{n-3}{3}$ generators in degree $3$, 
there are $2 \binom{n-2}{3}$ generators each in degree $4$
and in degree $5$, and there are 
$\binom{n-1}{3}$ generators in degree~$6$.
\end{theorem}

\section{Singularities and their Resolution}

In this section we study the singular locus  of the Kalman
variety $\mathcal{K}_{d,n} = \mathcal{K}_{d,n}(L)$, and we prove
that the map $p$ in (\ref{eq:twomoreprojections})
is a resolution of singularities. By definition, the {\em singular locus} ${\rm Sing}(\mathcal{K}_{d,n})$
is the subvariety of $\mathcal{K}_{d,n}$ defined by the
vanishing of the minors of size $n^2-n+d+1$ in the Jacobian matrix of
the generators of $I_{d,n}$. 

Returning to the setting of Section 2,  we consider the restricted incidence variety 
$$\, \mathcal{I} \cap q^*(\PP{{}}(L)) \quad = \quad \{(A,v) \in \PP^{n^2-1} \times \PP(L)\,|\,
v\textrm{\ is a eigenvector of\ }A\}\,. $$
This is a smooth variety because the fibers of the second projection $q$ to $\PP(L)$
are linear spaces of the same dimension. 
For $d=1$, the Kalman variety is a linear space, so it is smooth.
For $d \geq 2$, the Kalman variety $\mathcal{K}_{d,n}$
does have singularities:

\begin{lemma}\label{main2}
 For $2\le d\le n-1$ the singular locus of
$\mathcal{K}_{d,n}$ is given by the 
matrices $A$ such that there exists a subspace $L'\subseteq L$
of dimension $2$ which is $A$-invariant.
\end{lemma}

\begin{proof}
Consider a general pair $(A,v)\in \mathcal{I} \cap q^*(\PP{{}}(L))$.
Then $v$ is the only eigenvector of $A$ in $L$ (up to scalar multiples).
Therefore $\,p^{-1}(p(A,v))=p^{-1}(A)=\{(A,v)\}$, that is, the general fiber of
the map $p$ in (\ref{eq:twomoreprojections}) consists of only one point.
This means that the morphism $p$ is birational. On the other hand, if $A$ has  two 
linearly independent eigenvectors in $L$  then the fiber
$p^{-1}(A) $ consists of two points.

A standard result in algebraic geometry, known as {\em Zariski's Main Theorem},
states that the inverse image of a normal point under a birational projective morphism 
is connected. Hence the point $A$ is not normal in $\mathcal{K}_{d,n}$
when $|p^{-1}(A)|  \geq 2$. We infer that
 ${\rm Sing}(\mathcal{K}_{d,n})$ contains the set described in the
statement of Lemma~\ref{main2}.
This containment must be an equality because
being singular is a local property, and
any two points $A,A' \in \mathcal{K}_{d,n}$ in the complement of that set
are locally isomorphic. The latter property holds because
 $(A,v)$ and $(A',v')$ are locally isomorphic points
in $ \mathcal{I} \cap q^*(\PP{{}}(L)) $  if $v$ resp.~$v'$ is
the unique eigenvector of $A$ resp.~$A'$ in $L$.
\end{proof}

\begin{corollary}\label{desing}
The restricted incidence variety
$\mathcal{I} \cap q^*(\PP{{}}(L))$ is 
the desingularization of the Kalman variety $\mathcal{K}_{d,n}$.
The map $p$ in (\ref{eq:twomoreprojections}) is a
resolution of singularities.
\end{corollary}

\begin{proof}
The variety defined by the $2 \times 2$-minors of the matrix $(v, Av)$ is smooth.
If we impose the constraints $v \in L$, for a general linear subspace $L \subset \K^n$, 
then the intersection remains smooth, by Bertini's Theorem.
That intersection is precisely our incidence variety 
$\, \mathcal{I} \cap q^*(\PP{{}}(L))$. However, that variety is
independent of $L$.
\end{proof}

\begin{example}
If $d=2$ then ${\rm Sing}(\mathcal{K}_{2,n})$
consists precisely of the matrices $A$ whose
lower left $(n - 2)\times 2$-block $A_{21}$ is zero.
Hence ${\rm Sing}(\mathcal{K}_{2,n})  $ is a
 linear subspace of codimension  $2n-4$.
It contains the $\PP^{n^2-2n} $ over which 
 $\mathcal{K}_{2,n}$ is a cone, by Remark~ \ref{itsacone}.
 \end{example}

We extend our description of the singular locus 
of the Kalman variety by defining
$$\mathcal{K}_{s,d,n}\,=\,\left\{A \in \PP^{n^2-1} \,\,|\,
\textrm{\ there is a\ }A\textrm{-invariant subspace of dimension\ }\ge s
\textrm{\  in\ } L\right\} . $$

Note that $\mathcal{K}_{d,d,n}$ is a linear space of codimension $d(n-d)$,
consisting of the matrices $A$ whose
lower left $(n-d)\times d$-block $A_{21}$ is zero. We have the chain of inclusions
$\,\mathcal{K}_{d,d,n}\subseteq \mathcal{K}_{d-1,d,n}\subseteq\ldots\subseteq \mathcal{K}_{1,d,n}=\mathcal{K}_{d,n}$. The following theorem extends
Lemma \ref{main2}. It concerns the iterated singularity structure of the Kalman variety.

\begin{theorem}\label{codimkaldsn}
For  $1\le s\le d-1$, 
the variety $\,\mathcal{K}_{s+1,d,n}\,$ is precisely the singular locus of $\,\mathcal{K}_{s,d,n}$.
The codimension of $\,\mathcal{K}_{s,d,n}\,$ in $\,\PP^{n^2-1}\,$ is $\,s(n-d)$,
and  the codimension of $\,\mathcal{K}_{s,d,n}\,$ in $\,\mathcal{K}_{s-1,d,n}\,$ equals $n-d$.
In particular, ${\rm codim}\bigl({\rm Sing}(\mathcal{K}_{d,n})\bigr) = n-d$.
\end{theorem}

\begin{proof}
We generalize the incidence variety used in the proof of  Lemma \ref{main2}. Let
$$\mathcal{KS}_{s,d,n}\,\,= \,\, \{(A,L')\in \PP^{n^2-1}\times Gr(s,L) \, | \,
L'\textrm{\ is\ }A\textrm{-invariant}\}.$$
     This  is a smooth variety of dimension $n^2-1-s(n-d)$ because the second projection to the 
Grassmannian $Gr(s,L)$ is surjective and  its fibers 
are linear spaces of dimension $n^2-1-s(n-s)$.
The first projection $p$ is bijective when restricted to
$p^{-1}\left(\mathcal{K}_{s,d,n} \backslash \mathcal{K}_{s+1,d,n}\right)$
because a matrix $A\in \mathcal{K}_{s,d,n}\setminus\mathcal{K}_{s+1,d,n}$
determines uniquely its $A$-invariant subspace $L'$.
It follows that $\mathcal{KS}_{s,d,n}$ is the desingularization of $\mathcal{K}_{s,d,n}$.

A general matrix $A\in {K}_{s+1,d,n}$ has $s+1$ independent eigenvectors in $L$,
so  its fiber $p^{-1}(A)$ consists of ${{s+1}\choose s}=s+1$ distinct points.
Again by Zariski's Main Theorem, these are singular points of $\mathcal{K}_{s,d,n}$,
so that  ${\rm Sing}(\mathcal{K}_{s,d,n})$ contains $\mathcal{K}_{s+1,d,n}$.
However, this containment is an equality, by the argument used for  Lemma \ref{main2}.
\end{proof}

We next state a result about the defining equations of 
the iterated singular loci. 

\begin{theorem}\label{sing}
The variety  $\mathcal{K}_{s,d,n}$ is the zero set
of the $(n-s+1) {\times} (n-s+1)$-minors
of the Kalman matrix (\ref{eq:kalman}), or equivalently
the zero set of the $(d-s+1) {\times} (d-s+1)$-minors
of the reduced Kalman matrix (\ref{eq:kalman3}).
In particular,  ${\rm Sing}(\mathcal{K}_{d,n})$ 
consists of all matrices $A$ such that the reduced Kalman matrix has rank $\leq d-2$.
\end{theorem}

\begin{proof} We must show that an $n {\times} n$-matrix $A$ has an $s$-dimensional $A$-invariant subspace $L'\subseteq L={\rm ker}(C)$
if and only if the Kalman matrix (\ref{eq:kalman}) has rank $\le n-s$.
 We begin with the only-if direction.
 If $A$ has an $s$-dimensional invariant subspace $L'\subseteq L$,
then every $v\in L'$ satisfies $A^iv\in L'$,
which implies $CA^iv=0$. Hence $L'$ is contained in the kernel of the Kalman matrix (\ref{eq:kalman}),
which must have rank $\le n-s$.  Conversely, we prove that
the kernel of the Kalman matrix is $A$-invariant.
We assume that $L = \K \{e_1,\ldots,e_d\}$.
 Any vector $v$ in the kernel of the Kalman matrix has the form $v = \begin{pmatrix} w \\ 0 \end{pmatrix}$
with $A_{21} A_{11}^i w = 0 $ and $i = 0,1,\ldots,d-1$.
Then $Av=\begin{pmatrix} A_{11}w \\ 0 \end{pmatrix}$. This
vector is in the kernel of the Kalman matrix if $A_{21} A_{11}^i (A_{11}w) = A_{21} A_{11}^{i+1}w=0 $
for $i = 0,1,\ldots,d-1$. Since $A_{11}^dw$ is a linear combination of
$w,A_{11}w,\ldots A_{11}^{d-1}w$, because $A_{11}$ is a $d\times d$ matrix,  these equalities are satisfied.
This shows that the kernel of the Kalman matrix is $A$-invariant.
Hence $A\in\mathcal{K}_{s,d,n}$.
\end{proof}

A case of special interest is $s = d-1$. The variety $\mathcal{K}_{d-1,d,n}$ is
cut out by the $2 {\times} 2$-minors of the reduced Kalman matrix (\ref{eq:kalman3}).
Experimental evidence suggests that some of the nice combinatorics seen in Section 3
generalizes to this case. In particular, the ideal generated by these $2 {\times} 2$-minors
of (\ref{eq:kalman3}) appears to be prime,
 and  we conjecture that its initial ideal in the lexicographic term order is radical.

\smallskip

We close this section by stating a formula for the degrees of our iterated singular loci.
It will be proved in the next section. Our proof technique involves results from representation theory
that require  the field $\K$ to have characteristic $0$ (see \cite{Fu2}).

\begin{theorem}
\label{thm:degreesing}
Let ${\rm char}(\K) = 0$.
The degree of the subvariety
${\mathcal{K}}_{s,d,n}$ in $\, \PP^{n^2-1}$ is the coefficient of $ (x_1 x_2 \cdots
x_s)^{d-s}$ in
the expansion of the symmetric rational function
\begin{equation}
\label{eq:degformula1}
 \frac{\prod_{i=1}^s(1+x_i)^n}{\prod_{i,j=1}^s \bigl(1+(x_i-x_j) \bigr)}
 \end{equation}
 as a $\Z$-linear combination of Schur polynomials $\,s_\lambda(x_1,\ldots,x_s)$.
In particular, for fixed $d > s > 0 $,  the degree of the variety $\mathcal{K}_{s,d,n}$ is a polynomial
in $n$ of the form
\begin{equation}
\label{eq:degformula2}
 \deg ({\mathcal{K}}_{s,d,n}) \,\, \,=\,\,\,\,
\frac{\deg Gr(s,L)}{[s(d-s)]!} \cdot n^{s(d-s)}\,+\,\cdots\textrm{ (lower terms in $n$)},
\end{equation}
where $\,\deg Gr(s,L)=\frac{1!2!\ldots (s-1)![s(d-s)]!}{(d-s)!(d-s+1)!\ldots (d-1)!} \,$
is the degree of the Grassmannian.
\end{theorem}

\begin{example}
Consider the Kalman variety $\mathcal{K}_{3,5}$
in Example \ref{ex:fuenfdrei}.
Its singular locus $\mathcal{K}_{2,3,5}$
is the codimension $4$ variety defined
the $2 {\times} 2$-minors of the
$6 {\times} 3$-matrix in (\ref{ex:fuenfdrei}).
Here $ s=2, d=3, n=5$, and the symmetric rational function (\ref{eq:degformula1}) equals
$$ \frac{ (1 + x_1)^5 (1+x_2)^5}{(1-(x_1-x_2)^2)} \,= \, 1 + 5 (x_1 + x_2) +
{\bf 12}  x_1 x_2 \,+\, 11 (x_1^2+x_1x_2+x_2^2) + \hbox{higher terms.} $$
We see that $\mathcal{K}_{2,3,5} \subset \PP^{14}$ has degree $12$.
The lexicographic Gr\"obner basis for
$\mathcal{K}_{2,3,5}$ consists of $3$ quadrics,
$9$ cubics and $3$ quartics, with square-free initial terms. \qed
\end{example}

\section{Kalman varieties via vector bundles}

In this section we develop a geometric interpretation of the
Kalman variety  and its desingularization, in terms of
the tangent bundle of a suitable Grassmannian. This works also for the
varieties $\mathcal{K}_{s,d,n}$ introduced in the previous section.
This includes  $\mathcal{K}_{d,n}$ for $s=1$.
We shall assume from now on that $\K$ has characteristic $0$.

Let $\textrm{ad\ }\K^n$ denote the vector space of $n {\times} n$-matrices with trace zero.
We shall consider our varieties in the corresponding projective space $\PP^{n^2-2}=\PP(\textrm{ad\ }\K^n)$.
To be precise, we write
 $\tilde{\mathcal{K}}_{s,d,n}$ for the intersection of $\mathcal{K}_{s,d,n}$
with $\PP^{n^2-2}$. This is a subvariety of codimension $s(n-d)$.
We do not lose any information when passing from  $\mathcal{K}_{s,d,n}$ to $\tilde{\mathcal{K}}_{s,d,n}$,
because  $\mathcal{K}_{s,d,n}$ is a cone with vertex at the identity matrix
${\bf I}_{n \times n}$. In other words, the assumption ${\rm trace}(A) = 0$ is no restriction in generality because $A$ and 
$\,A - \frac{{\rm trace}(A)}{n} \cdot {\bf I}_{n \times n}$ have the same eigenvectors.
The Kalman variety is identified with
 $\tilde{\mathcal{K}}_{d,n}=\tilde{\mathcal{K}}_{1,d,n}$ which has codimension
$n-d$ and degree  $\binom{n}{d-1}$ in $\PP^{n^2-2}$.
 We shall derive the degree of
 $\tilde{\mathcal{K}}_{s,d,n}$, as promised in Theorem~\ref{thm:degreesing}.
 
Let $Gr(s,\K^n)$ be the $s(n-s)$-dimensional
Grassmannian of $s$-planes in $\K^n$. 
Let $S$ be the universal bundle of rank $s$ and $Q$ the quotient bundle of rank $n-s$.
They appear in the following exact sequence, with $\OO$  the structure sheaf of
$Gr(s,\K^n)$:
\begin{equation}\label{Euler}
0\rig{}S\rig{}\OO\otimes\K^n\rig{}Q\rig{}0
\end{equation}
The fiber of $TGr(s,\K^n)$ at $x=[V]$, where $V\subset\K^n$ is a $s$-dimensional subspace
of $\K^n$, is isomorphic to
$ Hom(V,\K^n/V)$. This means that $TGr(s,\K^n)=Hom(S,Q)$.
Every traceless matrix $A\in\textrm{ad}\ \K^n$ induces a section $s_A\in H^0(TGr(s,\K^n))$
which corresponds to the composition $V\rig{i}\K^n\rig{A} \K^n\rig{\pi} \K^n/V$
on the fiber of $[V]$. The section $s_A$ vanishes in $[V]$ if and only if 
$\pi(A(v))=0$ for all $v \in V$. This implies

\begin{lemma}\label{eigenvanish}
For $A \in$ \textrm{ad }$\K^n$, the section $s_A$ vanishes in $V$ if and only if 
$A(V)\subseteq V$.
\end{lemma}

By the Borel-Weil Theorem, all sections $s$ of the tangent bundle $TGr(s,\K^n)$
have the form $s=s_A$ for some $A\in \textrm{ad }\K^n$. Equivalently,
 $\,H^0(TGr(s,\K^n))=\textrm{ad }\K^n$.
Since the generic traceless $n\times n$ matrix has $n$ distinct eigenvectors, a
generic section of $TGr(s,\K^n)$ vanishes in ${n\choose s}$ points,
corresponding to the subspaces spanned by $s$ of the $n$ eigenvectors.
We define   $M^{s,n}$ to be the kernel of the evaluation map 
$$H^0(TGr(s,\K^n))\otimes\OO\,\rig{}\,TGr(s,\K^n).$$
Thus $M^{s,n}$  is a homogeneous bundle on $Gr(s,\K^n)$ of rank $n(n{-}s)+(s^2{-}1)$.
For $s=1$ we write $M^{1,n}=: M^n$. This is a homogeneous bundle of rank $n(n{-}1)$
on the projective space $\PP^{n-1}$. For example, on 
$Gr(1,\K^2) = \PP^1$ we have  $M^2=\OO(-1)^2$.

The Kalman variety of a subspace $L$ in $\K^n$ (for $s = 1$) and its  singular strata (for $s \geq 2$)
arise from the projectivization of the restriction of $M^{s,n}$ to $Gr(s,L)$.
Namely, we consider the product
$X=\PP{{}}(\textrm{ad\ }\K^n)\times Gr(s,\K^n)$
with its two projections
\begin{equation}\label{twoprojections}\begin{array}{ccccc}&&X\\
&\swar{p}&&\sear{q}\\
\PP{{}}(\textrm{ad\ }\K^n)&&&&Gr(s,\K^n)
\end{array}
\end{equation}

\begin{theorem}\label{desing2}The desingularization of the variety 
$\tilde{\mathcal{K}}_{s,d,n}$ 
is isomorphic to the projective bundle
$\PP{{}}(M^{s,n}_{|Gr(s,L)})$ where $Gr(s,L)\subset Gr(s,\K^n)$ 
parametrizes planes in~ $L$.

\end{theorem}
\begin{proof} 
Generalizing (\ref{projkalman}), we have
$\tilde{\mathcal{K}}_{d,s,n}(L) = p_*({\mathcal J}\cap q^*Gr(s,L))$, where
\begin{equation}
\label{eq:incidence}
{\mathcal J}=\{(A,V) \in \PP^{n^2-2} \times Gr(s,\K^n)\,|\,A(V)\subseteq V\} 
\end{equation}
is the incidence variety.
By the same argument as in Corollary~\ref{desing}, the subvariety
$\mathcal{KS}_{s,d,n}={\mathcal J}\cap q^*Gr(s,L)$ of $X$
coincides with the desingularization of $\tilde{\mathcal{K}}_{s,d,n}$ .
The fiber $q^{-1}(q(A,V))=q^{-1}(V)$ consists of all endomorphisms $A'$ 
that have $V$ as an invariant subspace.
By Lemma \ref{eigenvanish}, this is the set of all sections of the tangent bundle which vanish at $V$,
and hence it coincides with the fiber of $M^{s,n}$ at $V$.
\end{proof}

\begin{remark} After this paper had been written and refereed, we received Steven Sam's preprint \cite{Sam} 
which uses Theorem \ref{desing2} and the Weyman method to obtain the ideal sheaves
of ${\mathcal{K}}_{2,n}$ and ${\mathcal{K}}_{3,n}$, among other many results. 
In particular, Sam proved our conjecture (Theorem \ref{3generators})
on the defining equations of $\mathcal{K}_{3,n}$.
 \end{remark}

The restriction $S_{|Gr(s,L)}$ is the universal bundle on $Gr(s,L)$. We denote
this bundle again by $S$.
The restriction $Q_{|Gr(s,L)}$ splits as $Q_L\oplus\OO^{n-d}$, where $Q_L$ is the quotient bundle 
of rank $d-s$ on $Gr(s,L)$.
The bundle appearing in Theorem \ref{desing2} can be then described in the following way:

\begin{proposition} 
\label{prop:letssplit} We have
$\,M^{s,n}_{|Gr(s,L)}\,=\, M^{s,d}\oplus\OO^{n(n-d)}\oplus\left[(Q_L)^*\right]^{n-d} $.
\end{proposition}

\begin{proof}
This is derived from the following commutative diagram on $Gr(s,L)$:
{\small
$$\begin{array}{ccccccccc}
&&0&&0&&0\\
&&\dow{}&&\dow{}&&\dow{}\\
0&\!\!\!\!\rig{}\!\!\!&M^{s,d}&\!\!\!\rig{}\!\!\!&H^0(TGr(s,L))\otimes\OO&\!\!\!\rig{}\!\!\!&TGr(s,L)&\!\!\!\rig{}\!\!\!\!&0\\
&&\dow{}&&\dow{}&&\dow{}\\
0&\!\!\!\!\rig{}\!\!\!&M^{s,n}_{|Gr(s,L)}&\!\!\!\rig{}\!\!\!&\!\! H^0(TGr(s,\K^n))\otimes\OO\! \!&\!\!\!\rig{}\!\!\!&TGr(s,\K^n)_{|Gr(s,L)}\!&\!\!\!\rig{}\!\!\!\!&0\\
&&\dow{}&&\dow{}&&\dow{}\\
0&\!\!\!\!\rig{}\!\!\!&\!\OO^{n(n-d)}\oplus \left[(Q_L)^*\right]^{n-d}\! \!&\!\!\!\rig{}\!\!\!&\OO^{n^2-d^2}&\!\!\!\rig{}\!\!\!&(S^*)^{n-d}&\!\!\!\rig{}\!\!\!\!&0\\
&&\dow{}&&\dow{}&&\dow{}\\
&&0&&0&&0\\
\end{array}$$
}
The second and third columns split, and hence also the first column splits.
\end{proof}

The appearance of the trivial summand $\OO^{n(n-d)}$ corresponds to the reduction
(in the Introduction) from the Kalman matrix (\ref{eq:kalman}) to the small Kalman matrix~(\ref{eq:kalman2}).
In the case $(s,d)=(1,2)$ of Section 3,
the relevant Grassmannians are
 $Gr(1,L)=\PP^1\,\subset\, Gr(1,\K^n) = \PP^{n-1}$, and
 Proposition \ref{prop:letssplit} says
 $\,M^{n}_{\PP^1}=\OO(-1)^{n}\oplus\OO^{n(n-2)}$.
 
We recall (e.g.~from \cite{Fu}) that if a vector bundle $E$ of rank $r$ on a variety $X$ has a section
vanishing on $Z$, and the codimension of $Z$ is equal to $r$,
then the class of $[Z]$ in the degree $r$ component of the Chow ring $A^r(X)$ 
is computed by $[Z]=c_r(E)$.
We shall apply this to  the following vector bundle on the product variety $X$
in (\ref{twoprojections}):
$$ E \,\,:= \,\, p^*\OO(1)\,\otimes \, q^*TGr(s,\K^n). $$
By K\"unneth' formula,
$$ H^0(X,E)\,\,= \,\,H^0(\PP{{}}(\textrm{ad\ }\K^n),\OO(1))\,\otimes \, H^0(TGr(s,\K^n))
\,\,\simeq\,\,
{\textrm{End(ad\ }\K^n\textrm{)}} . $$
Hence the identify map on $\textrm{ad\ }\K^n$
determines a canonical section $c\in H^0(E)$.

\begin{proposition}\label{zerolocus} The zero locus $Z(c)$ of $c$ equals the incidence variety
$\mathcal{J}$ in (\ref{eq:incidence}).
\end{proposition}

\begin{proof} We shall compute $c$ on the fiber of $E$ over $(A,V)$.
The fiber is $\langle A\rangle\otimes Hom(V,\K^n/V)$
and the value of $c$ is $A\otimes \left(A_{|V}\colon V\to \K^n/V\right)$.
So $c$ vanishes exactly at the pairs $(A,V)$ such that $V$ is $A$-invariant. This proves the assertion.
\end{proof}

Since $Z(c)$ has codimension $\textrm{rk\ } E = s(n-s)$ in $X$ (e.g.~by Theorem \ref{codimkaldsn}),
 the class of $Z(c)$ equals the top Chern class of $E$.
In symbols, $\,\left[Z(c)\right]\,=\,c_{s(n-s)}(E)$.

\begin{theorem}\label{degkdns}
The degree of  ${\mathcal{K}}_{s,d,n}$
equals $\,c_{s(d-s)}(TGr(s,\K^n))\cdot \left[c_s(S^*)\right]^{n-d}$.
\end{theorem} 

\begin{proof}
The class of $Gr(s,L)\subset Gr(s,\K^n)$ is $\left[c_s(S^*)\right]^{n-d}$.
The projection $p$ in (\ref{twoprojections}) gives a birational map
$Z(c)\cdot q^*\left[c_s(S^*)\right]^{n-d}\to {\mathcal{K}}_{s,d,n}$.
The desired degree equals
\begin{equation}
\label{eq:whatwewant} {\rm deg}({\mathcal{K}}_{s,d,n}) \,\,=\,\,
p^*c_1(\OO(1))^{n^2-2-s(n-d)}\cdot c_{s(n-s)}(E)\cdot q^*\left(c_s(S^*)\right)^{n-d}.
\end{equation}
The Chern class of  $\, E \,= \, p^*\OO(1)\,\otimes \, q^*TGr(s,\K^n)\,$ decomposes as
$$c_{s(n-s)}(E) \,\,=\,
\bigoplus_{i=0}^{s(n-s)}p^*c_1(\OO(1))^{s(n-s)-i}\cdot q^*c_i(TGr(s,\K^n)).$$ 
Hence the quantity on the right of (\ref{eq:whatwewant}) can be written as
$$
\sum_{i=0}^{s(n-s)}p^*c_1(\OO(1))^{n^2-2+s(d-s)-i}\cdot q^* \! 
\left[c_i(TGr(s,\K^n))\cdot \left(c_s(S^*)\right)^{n-d}\right]$$ 
All summands are zero except for $i=s(d-s)$, and the result follows.
\end{proof}

\begin{proof}[Proof of Theorem~\ref{thm:degreesing}]
 The cohomology ring of the Grassmannian is identified with a
ring of symmetric polynomials (see e.g.~\cite{Fu2}), and under this identification we have
\begin{equation}
\label{eq:elemsym}  c_i(S^*) \ \longleftrightarrow \
e_i \,\,\,{\small\textrm{ $=$ $\,i$-th elementary symmetric function in\ } x_1,\ldots, x_s} 
\end{equation}
From the sequence (\ref{Euler}), tensored by $S^*$, we get the Chern polynomial
$$ p_{TGr(s,\K^n)}(t) \,\,=\, \, \frac{p_{S^*}(t)^{n}}{p_{S\otimes S^*}(t)}
\,\,=\,\, \frac{(\sum_{i=0}^sc_i(S^*)t^i)^n} {\sum_{i=0}^sc_i(S\otimes S^*)t^i } \,\,=\,\,
\frac{\prod_{i=1}^s(1+x_it)^n}{\prod_{i,j=1}^s \bigl(1+(x_i - x_j)t \bigr)} $$
Theorem \ref{degkdns} and (\ref{eq:elemsym}) imply
the result. Note the analogy with (3) in~\cite[\S 16]{BH}.
 \end{proof}

We end by describing a few special cases of Theorem~\ref{thm:degreesing}.
For $s=1$ we recover Proposition \ref{thm:degree} for the original
Kalman variety:
$\deg {\mathcal{K}}_{1,d,n}=\deg c_{d-1}T\PP^{n-1}={{n}\choose{d-1}}$.
For $s=d-1$, the Grassmannian $Gr(s,\K^d)$ is a projective space and 
the degree of ${\mathcal{K}}_{d-1,d,n}$
is the coefficient of $t^{d-1}$ in the rational function $(1+t)^d/(1-t)^{n-d}$.
For $s=2$, the Chern polynomial in 
the above proof is the appropriate truncation~of
$$p_{TGr(2,\K^n)}(t)
\,\,= \,\,\frac{(1+c_1(S^*)t+c_2(S^*)t^2)^n}{1-\left(c_1(S^*)^2-4c_2(S^*)\right)t^2}
\,\,=\,\, \frac{( 1+x_1 t)^n (1+x_2 t)^n }{1 - (x_1-x_2)^2 t^2}.
$$
Christoph Koutschan from RISC-Linz kindly helped us by means of
his software \cite{Kou} for holonomic summation.
Writing $(a)_n=\prod_{i=0}^{n-1}(a+i)$, he found the formula
\begin{equation}
\label{eq:koutschan}
\deg ({\mathcal{K}}_{2,d,n}) \,\,=\,\,
(-1)^d \frac{2^{2d-3}}{(d-1)!}
 \sum_{k=0}^{d-2} \frac{(1/2-k)_{d-1} (n+1-k)_k
(d+n-2k)_k}{(2k)!}.
\end{equation}
This expression is  the degree of the singular locus of the Kalman variety $\mathcal{K}_{d,n}$.

We note that the series expansion of (\ref{eq:degformula1}) in terms of Schur polynomials
can be computed using Stembridge's {\tt Maple} package {\tt SF} \cite{Stem} or the
more recent package {\tt Schubert2} in {\tt Macaulay2} \cite{M2}.
  For our computations we used the latter, by means of the following  convenient function,
 which implements the formula in Theorem \ref{degkdns}:
 
 \smallskip
 
\begin{verbatim}
loadPackage "Schubert2"
kal = (s,d,n) -> (G = flagBundle({s,n-s}); (S,Q) = G.Bundles;
integral(((chern(s,dual(S)))^(n-d))*(chern(s*(d-s),(dual(S))**Q))))
\end{verbatim} 

\smallskip

This little piece of code highlights the utility of vector bundle techniques
as a practical tool for hands-on computations
concerning problems in linear algebra.

\bigskip

\end{document}